\documentclass{amsart}
\usepackage[utf8]{inputenc}

\usepackage{amsfonts}
\usepackage{amsmath}
\usepackage{amsthm}
 
\usepackage[all]{xy}
\usepackage{appendix}

\theoremstyle{definition}
\newtheorem{defi}{Definition}[subsection]

\newtheorem{rem}[defi]{Remark}
\newtheorem{prop}[defi]{Proposition}
\newtheorem{lem}[defi]{Lemma}
\newtheorem{theo}[defi]{Theorem}
\newcommand{\AFun}{\operatorname{A}_{\infty}\operatorname{Fun}^{\circ}}
\newcommand{\A}{\mathcal{A}}
\newcommand{\B}{\mathcal{B}}
\newcommand{\dgcat}{\mathsf{DGCat}(k)}

\newcommand{\co}{\colon\thinspace}

\title{A note on a Holstein construction}

\author{Sergey Arkhipov}

\author{Daria Poliakova}

\begin{document}

\begin{abstract}
We clarify details and fill certain gaps in the construction of a canonical Reedy fibrant resolution for a
constant simplicial DG-category due to Holstein. 
\end{abstract}

\maketitle

\tableofcontents

\section{Introduction}
The present paper grew out of attempts to understand technical details of a proof in \cite{Hols}. Thus, from the very start, we do not claim that our work contains original  insights.\\

We begin by describing our interest. In the papers \cite{BHW}, \cite{AO} homotopy gluing of DG-categories was studied. \\

The standard example is given by Abelian categories of sheaves 
on open sets for a {\v C}ech covering of a topological space. One seeks a lift for gluing of Abelian categories  to DG-level. Unlike with ordinary categories, one requires coherence data on multiple intersections in the covering to be given by weak equivalences, not by isomorphisms. The answer is spelled out naturally in the language of homotopy limits for cosimplicial diagrams of DG-categories. In \cite{AO} Sebastian {\O}rsted  and the first author provided an explicit model for such a homotopy limit. \\

The construction relies on an explicit model for {\em{powering}}
by simplicial sets in the model category of DG-categories due to Holstein (see \cite{Hols}, Proposition 3.6). The key ingredient in the latter is a canonical simplicial resolution of a DG-category introduced in the same paper (see \cite{Hols}, Propositions 3.9 and 3.10). Our goal in the present paper is to add details to the sketch of the proofs of those statements in Holstein's work. \\

The author's strategy in that paper was to generalize a proof of Tabuada that a certain explicit DG-category
provides a {\em{path object}} construction (see \cite{Tab2}, Proposition 3.3). However, the original proof of Tabuada had some details omitted,  which led to a flaw in Holstein's approach. We fill the gap, and this, together with certain explicit calculations, is the main content of the present note.\\

Let us outline the structure of the paper. In the second section we recall the construction
of Dwyer-Kan model structure on the category of DG-categories. Then, following Lef{\`e}vre-Hasegawa \cite{LH} and Faonte \cite{faonte},  we discuss
 close relatives of DG-functors called $A_\infty$-functors. We describe the category of $A_\infty$-functors between two DG-categories playing the role of internal Hom in the category of DG-categories.
We conclude the section by recalling  the Reedy model structure on  a diagram category with values in a model category. In particular this includes our main object of interest -- the category of simplicial DG-categories. \\

In the third section we provide a detailed proof of Holstein's theorem filling the gap in his original approach. In particular, the proof of fibrancy of matching maps is given by explicit lifts. \\

Our proof is based on direct calculations of  lifts and on the use of an elegant description for homotopy equivalences of  $A_\infty$-functors suggested to us by Efimov. In the appendix, we provide an alternative approach to the proof developing the ideas of Tabuada and Holstein. The
main strategy there is to reduce the statement to the case of pretriangulated DG-categories via the construction of pretriangulated envelope.
\vskip 2mm
\noindent 
{\bf Acknowledgements.} The gaps in the last part of the paper \cite{Hols} were noticed by several people, in particular, by Boris Shoikhet. We thank Boris Shoikhet for sharing his concerns at an early stage of the present work. We also thank Julian Holstein for stimulating discussions. \\

The idea to study pointwise homotopy equivalences of A-infinity functors is 
due to Alexander Efimov, the context and the exact reference were kindly provided by Sebastian {\O}rsted. This improvement clarified and simplified the
exposition greatly, thus the final form of the paper owes a lot to Efimov and {\O}rsted. We thank Edouard Balzin for careful proofreading of the text, and Timothy Logvinenko for useful comments. \\

The first author was partially supported by QGM. The second author was partially supported by Laboratory of Mirror Symmetry NRU HSE, RF Government grant, ag. N 14.641.31.0001. The second author was also supported by the Danish National Research Foundation through the Centre for Symmetry and Deformation (DNRF92).

\section{Homotopy theory of DG-categories}
Below we collect a few constructions and statements to be used in the next section and necessary
to formulate the theorem of Holstein. We work over a base field $k$. Recall that a DG-category is a category enriched over the monoidal category ${\mathsf{Com}}(k{\mathsf{-Mod}})$. The homotopy category for a DG-category $\mathcal A$ is denoted by $H^0({\mathcal{A}})$.
We denote the category of small DG-categories and DG-functors by $\dgcat$.
\subsection{Dwyer-Kan model structure for DG-categories}

Recall that a DG-functor is called a quasiequivalence, if it induces quasiisomorphisms on all Hom complexes and becomes an equivalence of the homotopy categories. Quasiequivalences are a part of {\it Dwyer-Kan} model structure on ${\mathsf{DGCat}}(k)$  constructed in \cite{Tab}. Recall the description of the three standard classes of morphisms.\\

We say that a DG-functor $F\co \A \to \mathcal{D}$ is

\begin{itemize}

\item a weak equivalence, if it is a quasiequivalence

\item a fibration, if it is surjective on all $\operatorname{Hom}$ complexes and is an isofibration at the level of $H^0$, i.e. for a homotopy equivalence $F(x) \xrightarrow{u} y$ in $\mathcal{D}$ there exists a homotopy equivalence $x \xrightarrow{u'} y'$ such that $F(u') = u$:

$$\xymatrix{\A \ar[d]^-F & x \ar@{|-->}[d] \ar@{-->}[r]^{u'} & y' \ar@{|->}[d] \\
 \mathcal{D} & F(x) \ar[r]^{u} & y } $$

\item a cofibration, if it admits the left lifting property with respect to all trivial fibrations.
\end{itemize}

\begin{theo}
The category ${\mathsf{DGCat}}(k)$ is equipped with cofibrantly generated model structure with weak equivalences, fibrations and cofibrations defined as above. 
\end{theo}

\subsection{$A_\infty$ functors as inner $\operatorname{Hom}$} In $\dgcat$, one can take the naive tensor product $\A \otimes \mathcal{D}$ and the naive inner $\operatorname{Hom}$ $\operatorname{DGFun}(\A,\mathcal{D})$ which make $\dgcat$ into a closed monoidal category. However, these notions are not consistent with the model structure discussed above, and thus do not make $Ho\dgcat$ into a closed monoidal category. This can be amended by considering derived versions, $\otimes^L$ and $\operatorname{RHom}$ (see \cite{toen}),  which are defined up to quasiequivalence but which make $Ho\dgcat$ into a closed monoidal category. \\

Of existing models for $\operatorname{RHom}$, we make use of the one given by the DG-category of $A_\infty$-functors.

\begin{defi}
For two DG-categories $\A$, $\B$, a strictly unital $A_\infty$ functor $F\co \A \to \B$ consists of the following data:

\begin{itemize}
    \item $F_0: Ob\A \to Ob\B$
    \item for all $n \geq 1$ and $x_0, \ldots, x_n \in Ob\A$, $$F_n: \A(x_{n-1},x_n) \otimes \ldots \otimes \A(x_0,x_1) \to B(F_0(x_0),F_0(x_n))$$
    of degree $1-n$, subject to 
\begin{gather*} \sum_{s=0}^{n-2} (-1)^{n-s} F_{n-1}(Id^{\otimes s} \otimes m \otimes Id^{\otimes (n-s-2)}) \\ + \sum_{s=0}^{n-1} (-1)^{n-1} F_n (Id^{\otimes s} \otimes d \otimes Id^{\otimes(n-s-1) }) \\ = d F_n + \sum_{s=1}^{n-1}(-1)^{ns}m(F_s \otimes F_{n-s}), 
\end{gather*}

where $d$ is the differential and $m$ is the composition.
\end{itemize}
\end{defi}

\begin{defi}
For two DG-categories $\A$, $\B$, the DG-category $A_\infty\operatorname{Fun}(\A,\B)$ has strictly unital $A_\infty$ functors as objects. For $F$, $G$ being such, the complex $A_\infty\operatorname{Fun}(\A,\B)(F,G)$ is, in degree $l$,
$$ \prod_{\substack{n \geq 0 \\ x_0, \ldots, x_n \in Ob(\A)}} \operatorname{Hom}(\A(x_{n-1},x_n) \otimes \ldots \A(x_0,x_1), \B(F_0(x_0),G_0(x_n))[l-n])$$ 
For $a \in A_\infty\operatorname{Fun}^l(\A,\B)(F,G)$, its differential $d_{A_\infty}(a)$ has its component at $(x_0, \ldots, x_n)$ equal to 
\begin{gather*}
    \pm d(a_{x_0, \ldots, x_n})+ \sum_{i=1}^n\pm m(a_{x_i, \ldots, x_n} \otimes F_{x_0, \ldots, x_n}) \\ + \sum_{i=0}^{n-1} \pm m(G_{x_i, \ldots, x_n} \otimes a_{x_0, \ldots, x_i}) \\ + \sum_{i=0}^{n-1} \pm a_{x_0, \ldots, x_n}(Id^{\otimes i } \otimes d \otimes Id^{\otimes (n-i-1)})\\ + \sum_{i=0}^{n-2} \pm a_{x_0, \ldots, \widehat{x_i}, \ldots, x_n}(Id^{\otimes i }\otimes m \otimes Id^{\otimes (n-i-2)}) 
\end{gather*}

\end{defi}

The definitions above are a special case of the general theory of $A_\infty$ categories and their morphisms. The discussion in full generality and including sign conventions can be found e.g. in \cite{LH}. \\

In \cite{faonte}, the following theorem is proved. 

\begin{theo}
The DG-category $A_\infty\operatorname{Fun}(\A,\B)$ is a model for $\operatorname{RHom}(\A,\B)$.
\end{theo}

\subsection{Reedy model structure for diagrams}
To talk about (co)simplicial DG-categories, we need the following technique (see \cite{Hir} or \cite{Hov}).

\begin{defi}

A Reedy category is a category $\mathcal{I}$ together with a degree function $d\co Ob(\mathcal{I}) \to \lambda$ (where $\lambda$ is an ordinal, typically $\mathbb{N}$) and with two full subcategories $\mathcal{I}^+$ and $\mathcal{I}^-$, subject to the following conditions:

\begin{itemize}

\item every non-identity map in $\mathcal{I}^+$ increases the degree;

\item  every non-identity map in $\mathcal{I}^-$ decreases the degree;

\item every map $f$ in $\mathcal{I}$ admits a unique factorization $f = f^+ \circ f^-$, where $f^- \in \mathcal{I}^-$ and $f^+ \in \mathcal{I}^+$.

\end{itemize}

\end{defi}

%For example, every direct category is a Reedy category with $\mathcal{I}^+=\mathcal{I}$ and %$\mathcal{I}^-$ consisting of identities, and every inverse category is a Reedy category with %$\mathcal{I}^-=\mathcal{I}$ and $\mathcal{I}^+$ consisting of identities. 

The {\em simplicial category} $\Delta$  of finite ordinals and order preserving maps is an example of a Reedy category  -- in its case, $d([n])=n$, $\Delta^+$ consists of injections and $\Delta^-$ consists of surjections. Also, for $\mathcal{I}$ a Reedy category, $\mathcal{I}^{\mathrm{op}}$ is also a Reedy category with the same degree function, with $(\mathcal{I}^{\mathrm{op}})^+ = (\mathcal{I}^-)^{\mathrm{op}}$ and with $(\mathcal{I}^{\mathrm{op}})^-=(\mathcal{I}^+)^{\mathrm{op}}$. \\

For a Reedy category $\mathcal{I}$ and an arbitrary model category $\mathcal{M}$, the diagram category $\mathcal{M}^{\mathcal{I}}$ is equipped with Reedy model structure. We need the following definitions to describe it.

\begin{defi} 

\begin{enumerate}

\item For $i \in \mathcal{I}$, the latching category $\delta(I^+\downarrow i)$ is a full subcategory of the overcategory $(\mathcal{I}^+\downarrow i)$ consisting of all arrows except for $id_i$.

\item For $i \in \mathcal{I}$ and $D \in \mathcal{M}^{\mathcal{I}}$, the corresponding latching object is

$$L_iD = \operatornamewithlimits{colim}_{j \to i \in \delta(I^+\downarrow i)}D(j). $$

\item Dually, for $i \in \mathcal{I}$, the matching category $\delta(i\downarrow \mathcal{I}^-)$ is a full subcategory of the undercategory $(i\downarrow \mathcal{I}^-)$ consisting of all arrows except for $id_i$.

\item For $i \in \mathcal{I}$ and $D \in \mathcal{M}^{\mathcal{I}}$, the corresponding matching object is

$$M_iD = \lim_{i \to j \in \delta(i\downarrow \mathcal{I}^-)}D(j). $$

\end{enumerate}

\end{defi}
\noindent
Note that there are natural maps $L_iD \xrightarrow{l_iD} D(i) \xrightarrow{m_iD} M_iD$, and, for a map of diagrams $f\co D \to D'$, maps $L_i(f)\co L_iD \to D_iD'$ and $M_i(f)\co M_iD \to M_iD'$. 
Let us say that a map of diagrams $f\co D \to D'$ is

\begin{itemize}

\item a Reedy weak equivalence, if $\forall i \in \mathcal{I}$ the map $f_i\co D(i) \to D'(i)$ is a weak equivalence in $\mathcal{M}$

\item a Reedy cofibration, if $\forall i \in \mathcal{I}$, the arrow $$l_if\co L_i(D') \coprod\limits_{L_iD}D(i) \to D'(i)$$ is a cofibration in $\mathcal{M}$:

$$\xymatrix{L_i(D) \ar[r]^{L_i(f)} \ar[d]_{l_iD} & L_i(D') \ar[d] \ar@/^1.3pc/[ddr]^{l_iD'} & \\
D(i)\ar[r] \ar@/_1.3pc/[rrd]_{f_i} & L_i(D') \coprod\limits_{L_iD}D(i) \ar@{-->}[rd]^{l_if}&  \\
& & D'(i)}$$

\item a Reedy fibration, if $\forall i \in \mathcal{I}$ the arrow $$m_if\co D(i) \to M_i(D) \underset{M_i(D')}{\times} D'(i)$$ is a fibration in $\mathcal{M}$:

$$\xymatrix{D(i) \ar@/^1.3pc/[rrd]^{f_i} \ar@/_1.3pc/[rdd]_{m_iD} \ar@{-->}[rd]^{m_if} & & \\
& M_i(D) \underset{M_i(D')}{\times} D'(i) \ar[d] \ar[r] &D'(i)\ar[d]^{m_iD'} \\
&M_i(D)\ar[r]^{M_i(f)} & M_i(D')} $$

\end{itemize}
\begin{theo}
 The three classes of morphisms  define  a model structure on the category $\mathcal{M}^\mathcal{I}$.
\end{theo}

One notices that for $0 \to D$ the Reedy cofibrancy condition boils down to the cofibrancy of $L_i(D) \to D(i)$ for every $i$, and, dually, for $D \to 1$ the Reedy  fibrancy condition boils down to the fibrancy of $D(i) \to M_i(D)$ for every $i$. Thus a diagram is Reedy cofibrant if all its latching maps are cofibrations, and a diagram is Reedy fibrant if all its matching maps are fibrations.\\

In this note, our source category is $\Delta^{\mathrm{op}}$, and our target category is ${\mathsf{DGCat}}(k)$ with Dwyer-Kan model structure.

\section{Reedy fibrant replacement for simplicial DG-categories}
\subsection{Holstein construction}
Denote the DG-category obtained by  the $k$-linearization of the category for the totally ordered set $\{0,\ldots,n\}$ by $k[n]$. \\

For a DG-category $\A$, the DG-category $\AFun(k[n],\A)$ has $A_\infty$ functors $k[n] \to \A$ sending arrows to homotopy equivalences as objects and the complexes of $A_\infty$ natural transformations as morphisms. We spell out the formulas for our case. An object $(X,f) \in \AFun(k[n],\A)$ is the data of $(n+1)$ objects $X_0$, $\ldots$, $X_n$ in $\A$ and the morphisms $\{ f_I \}$ where $I$ runs over all subsets of $\{0, \ldots, n\}$ of cardinalities at least 2, with $f_{i_0, i_1, \ldots, i_k} \in \A^{1-k}(X_{i_0},X_{i_k})$, subject to the following conditions:

\begin{itemize}
    \item $d(f_{i_0,\ldots,i_k}) = \sum_{s=1}^{k-1} (-1)^{s} f_{i_0, \ldots, \widehat{i_s}, \ldots, i_k} - \sum_{s=1}^{k-1} (-1)^{s} f_{i_s, \ldots, i_k} \circ f_{i_0, \ldots, i_s}$ 
    \item all $f_{i,j}$ are homotopy equivalences.
\end{itemize}

Following Holstein, we use the following notation:

\begin{itemize}
    \item $d(\phi)_{i_0, \ldots, i_k} = d(\phi_{i_0, \ldots, i_k})$
    \item $(\Delta \phi)_{i_0,\ldots,i_k} = (-1)^{|\phi|}\sum_{s=1}^{k-1} (-1)^s \phi_{i_0, \ldots, \widehat{i_s}, \ldots, i_k}$
    \item $(\phi \circ \psi)_{i_0, \ldots, i_k} = \sum_{s=0}^{k} (-1)^{|\phi|s} \phi_{i_s, \ldots, i_k} \circ \psi_{i_0, \ldots, i_s}$, where one should read 0 if indexing subset is impossible.
\end{itemize}

In this notation, upon fixing $|f|=1$, the first of the conditions above becomes Maurer-Cartan equation: 

$$ d(f)+\Delta f + f \circ f = 0. $$

The $\operatorname{Hom}$ complexes in  $\AFun(k[\bullet],\A)$ are the complexes of $A_\infty$ natural transformations, namely

$$\AFun(k[\bullet],\A)((X,f),(Y,g)) = \bigoplus_{\{i_0,\ldots, i_k \} \subset \{0, \ldots, n\}}\A(X_{i_0},Y_{i_k})[-k] $$

with differential 

$$ d_{A_\infty}(a) = d(a)+ \Delta a + a \circ f - (-1)^{|a|}g \circ a.$$

Explicitly, a degree $m$ morphism $a\co (X,f) \to (Y,g)$ consists of components $\{ a_I \}$ where $I$ runs over all non-empty subsets of $\{0, \ldots, n\}$, with $a_{i_1, \ldots, i_k} \in \A^{1-k}(X_{i_1},Y_{i_k})$.  \\

As $k[\bullet]$ is a cosimplicial DG-category,  $\AFun(k[\bullet],\A)$ becomes a simplicial DG-category, with structure maps obtained by precompositions with structure maps of $k[\bullet]$. \\

One of the main results in the paper \cite{Hols} is the following theorem (see Propositions 3.9 and 3.10 in that paper). 

\begin{theo}
\label{main}
The simplicial DG-category $\AFun(k[\bullet],\A)$ as an object of $\mathsf{DGCat}(k)^{\Delta^{\operatorname{op}}}$ is  a Reedy fibrant replacement of $c\A$, the constant simplicial DG-category for a DG-category $\A$, with respect to Dwyer-Kan model structure on the target model category $ \mathsf{DGCat}(k)$.
\end{theo}

For convenience, we denote $\AFun(k[\bullet],\A)=:F_\bullet(\A).$ \\

The proof naturally consists of two parts. Firstly, one has to show that for every $n$, the natural (constant functor) inclusion $\A \to F_n(\A)$ is a quasiequivalence.  Secondly, one has to show that $F_\bullet(\A)$ is Reedy fibrant. \\

\subsection{Quasiequivalences}
In both parts of the proof, we rely on the following general fact from the homotopy theory of $A_\infty$-functors, due to Lef{\`e}vre-Hasegawa, Proposition 8.2.2.3 in \cite{LH}. We reduce the generality by considering DG-categories instead of $A_\infty$-categories.

\begin{lem}
\label{technical}
Let $\mathcal{A}$, $\mathcal{B}$ be two DG-categories, $F$, $G$ two $A_\infty$-functors $\mathcal{A} \to \mathcal{B}$ and $a\co F \to G$ a closed $A_\infty$ natural transformation of degree 0. Then $a$ is a homotopy equivalence in $A_\infty \operatorname{Fun}(\mathcal{A},\mathcal{B})$ if and only if for every $X \in \mathcal{A}$ the component $a_X\co F(X) \to G(X)$ is a homotopy equivalence in $\mathcal{B}$. 
\end{lem}

Note if the DG-category $A_\infty \operatorname{Fun}(\mathcal{A},\mathcal{B})$ is replaced by the ``naive version of inner $\operatorname{Hom}$"  $ \operatorname{DGFun}(\mathcal{A},\mathcal{B})$, then the statement of the lemma above would not hold. \\ 

We can now prove the following theorem.

\begin{theo}
For every $n$, the constant functor inclusion $c\co \A \to F_n(\A)$ is a quasiequivalence.
\end{theo}
\begin{proof}
We first check that $c$ induces quasiisomorphism on all $\operatorname{Hom}$ complexes. It is injective on cohomology -- if for $f\co X \to Y$ we have $cf = d_{A_\infty}(g)$, then in particular $f = (cf)_0 = d(g_0)$. To show that $c$ is surjective on cohomology, let $a$ be a closed map $cX \to cY$ for $X,Y \in \A$. Let us check that $a$ is in the same cohomology class as $c(a_0)$, i.e. that $a-c(a_0)$ is exact. The fact that $d_{A_\infty}(a) = 0$ corresponds to the following formulas: 

$$ \begin{cases} d(a_i) = 0 \\ d(a_{i_0 \ldots i_k}) = a_{i_1 \ldots i_k} - a_{i_0 \ldots i_{k-1}} + \sum_{s=1}^{k-1}(-1)^s a_{i_0 \ldots \widehat{i_s} \ldots i_k} \end{cases}$$

Then $a-c(a_0) = d(b)$, where 

$$b_{i_0 \ldots i_k} = \begin{cases} 0 & i_0 = 0 \\
a_{0 i_0 \ldots i_k} & i_0 \neq 0 \end{cases}$$

We then  check that $c$ is essentially fully faithful at the level of $H^0$, namely that any object $(X,f) \in F_n(\A)$ is homotopy equivalent to an object in the image of $c$. Indeed, consider the object $cX_0$. The $A_\infty$-natural transformation $a\co cX_0 \to (X,f)$ is given by 

$$a_{i} = \begin{cases} 1_{X_0} & i = 0 \\ f_{0i} & \mathrm{otherwise}\end{cases} $$ 
$$a_{i_0 \ldots i_k} = \begin{cases} 0 & i_0 = 0 \\ f_{0 i_0 \ldots i_k} & \mathrm{otherwise}\end{cases}$$

The fact that $d_{A_\infty}(a) = 0$ follows from Maurer-Cartan condition for $f$.\\

Note that $1_{X_0}$ and $f_{0i}$ are all homotopy equivalences in $\A$. Then, by Lemma $\ref{technical}$, $a$ is a homotopy equivalence.
\end{proof}

\begin{rem}
In \cite{Hols}, it was fist shown that every $(X,f) \in F_n(\A)$ can be strictified, i.e. it is homotopy equivalent to an $(\tilde{X},\tilde{f})$ where all compositions are strict and $\tilde{f}_{i_0 \ldots i_k} = 0$ for $k>1$. However, Lemma \ref{technical} does not become elementary even in this generality, and once we have this lemma, strictification becomes unnecessary. 
\end{rem}

\subsection{Reedy fibrancy}
We now prove Reedy fibrancy of $F_\bullet (\A)$ by showing that the matching maps are Dwyer-Kan fibrations -- namely, that they are surjective on all the $\operatorname{Hom}$ complexes and that they are isofibrations at the level of $H^0$. We begin from explicitly describing these matching maps. \\

By definition of a matching object, we have 

$$M_nF(\A) = \lim_{\delta([n] \downarrow (\Delta^{\mathrm{op}})_{-})}F_\bullet(\A) = \lim_{[m]\hookrightarrow [n]}F_m(\A).$$

This is the data of $A_\infty$ functors without highest homotopies. Namely, an object $(X,f) \in M_nF(\A)$ is the data of $(n+1)$ objects $X_0$, $\ldots$, $X_n$ in $\A$ and the morphisms $\{ f_I \}$ where $I$ runs over all subsets of $\{0, \ldots, n\}$ of cardinalities from $2$ to $n$ (that is, the subset $\{0, \ldots, n\}$ is not included) with $f_{i_0, i_1, \ldots, i_k} \in \A^{1-k}(X_{i_0},X_{i_k})$, satisfying the following conditions:\\

\begin{itemize}
    \item $d(f) + \Delta f + f \circ f = 0$;
    \item all $f_{i,j}$ are homotopy equivalences.\\
\end{itemize}

Similarly, the morphisms are given by complexes of $A_\infty$ natural transformations without highest homotopies. Namely, a degree $m$ morphism $a\co (X,f) \to (Y,g)$ is the set of morphisms $\{ a_I \}$ where $I$ runs over all non-empty subsets of $\{0, \ldots, n\}$ except for $\{0, \ldots, n \}$ itself, with $a_{i_1, \ldots, i_k} \in \A^{1-k}(X_{i_1},Y_{i_k})$, and with differential given by 
$$d_{A_\infty}(a) = d(a)+ \Delta a + a \circ f - (-1)^{|a|}g \circ a.$$ 

The matching map $m_n\co F_n(\A) \to M_nF(\A)$ is the natural forgetful functor that, on objects, forgets $f_{0,1,\ldots,n}$, and, on morphisms, forgets $a_{0, 1, \ldots, n}$. We write $(X,f) \mapsto (X,f_{\leq n})$.  \\

The first part of Reedy fibrancy for $F_\bullet(\A)$ is the following elementary proposition.

\begin{prop}
The forgetful functor $m_n$ is surjective on Hom complexes.
\end{prop}

\begin{proof} 
A preimage of a truncated $A_\infty$ transformation $a$ between $(X,f_{\leq n})$ and $(Y,g_{\leq n})$ can be obtained by simply assigning any value (e.g. 0) to $a_{0,1, \ldots, n}$, as there are no conditions on the components. 
\end{proof}

Showing that $m_n$ is a homotopy isofibration requires more work. In our computations, we use the following lemma,  from \cite{Kon}, Section 5, Theorem 1 (see also \cite{Shoi}, Lemma 3.6).

\begin{lem}
\label{klemma}
For any DG-category $\A$ and a homotopy equivalence $f \in \A^0(X,Y)$  it is always possible to find $f \in \A^0(Y,X)$, $r_X \in \A^{-1}(X,X)$, $r_Y \in \A^{-1}(Y,Y)$ and $r_{XY} \in \A^{-2}(X,Y)$ such that: 
\begin{itemize}
    \item $gf = 1_X + d(r_X)$
    \item $fg = 1_Y + d(r_Y)$
    \item $fr_X - r_Yf = d(r_{XY})$
\end{itemize}
\end{lem}

Now suppose that we have an object $(Y,g) \in M_nF(\A)$ and a homotopy equivalence $a\co (X,f_{\leq n}) \to (Y,g)$ (with homotopy inverse $\overline{a}$). To show that $m_n$ is an isofibration on $H^0$, we need to lift $a$ to a homotopy equivalence in $F_n(\A)$. \\

\begin{rem}
In \cite{Hols}, the lift of the object is constructed -- namely, $g_{0,\ldots,n}$ is given with $d(g_{0,\ldots,n}) = (\Delta g +g\circ g)_{0,\ldots,n}$. We insignificantly modify the lift and provide the computation for the sake of reader's convenience. In what follows, let $\alpha \circ ' \beta$ denote $\alpha \circ \beta$ without the term $\alpha_{0, \ldots, n} \circ \beta_{0}$. Let $r_{Y_0}$ be such that $a_0\overline{a_0} = 1_{Y_0}+d(r_{Y_0})$. The indexing subset is always $\{0,1, \ldots, n \}$ and is omitted.
\end{rem}
\begin{prop}
Setting  $$g_{0,\ldots, n}\co = (\Delta a + a \circ f - g \circ'a )\overline{a_0} - (\Delta g + g\circ g) r_{Y_0}$$ indeed gives $d(g_{0,\ldots, n}) =  \Delta g + g\circ g$, thus this lifts the object.
\end{prop}

\begin{proof}
One first checks that $d(\Delta g + g\circ g)=0$. 
Then $$d((\Delta g + g\circ g) r_{Y_0}) = (\Delta g + g\circ g)d(r_{Y_0}) = (\Delta g + g\circ g) (a_0\overline{a_0}-1).$$
So we are left to see that $$d(\Delta a + a \circ f - g \circ'a)\overline{a_0}') = (\Delta g + g\circ g)a_0\overline{a_0}$$ -- or that $$d(\Delta a + a\circ f - g \circ'a) = (\Delta g + g\circ g)a_0,$$ which is an explicit computation.
\end{proof}

We now construct the {\it closed} lift of the morphism $a$ -- namely, we give a formula for $a_{0,\ldots,_n}$ with $d(a_{0,\ldots,_n}) = (\Delta a + a \circ f - g \circ a)_{0,\ldots,n}$. Let $r_{X_0}$ be such that $\overline{a_0}a_0 = 1_{X_0} + d(r_{X_0})$, let $r_{Y_0}$ be such that $a_0\overline{a_0} = 1_{Y_0} + d(r_{Y_0})$, and let $r_{X_0 Y_0}$ be such that $a_0r_{X_0}-r_{Y_0}a_0 = d(r_{X_0 Y_0})$ (such $r_{X_0}$, $r_{Y_0}$ and $r_{X_0 Y_0}$ can always be found due to Lemma \ref{klemma}). The indexing subset is again $\{0,1, \ldots, n \}$ and is omitted.

\begin{prop}
\label{lift}
Setting  $$a_{0,\ldots, n}\co = (\Delta a + a \circ f - g \circ ' a) r_{X_0} + (\Delta g+ g \circ g) r_{X_0Y_0}$$ indeed gives $d(a_{0,\ldots, n}) = \Delta a + a \circ f + g \circ a$, thus this lifts the morphism.
\end{prop}

\begin{proof}
We start from observing that $g \circ a = g \circ ' a +  g_{0,\ldots,_n}a_0$ and we can insert our value of $g_{0,\ldots,n}$. This gives 
$$ \Delta a + a \circ f - g \circ a = \Delta a + a \circ f - g \circ' a -  (\Delta a + a \circ f- g \circ'a)\overline{a_0}a_0 + (\Delta g + g\circ g) r_{Y_0}a_0. $$

We know that $d(\Delta a + a \circ f - g \circ'a) = (\Delta g + g\circ g)a_0$, so
$$d((\Delta a + a \circ f - g \circ ' a) r_{X_0})=(\Delta g + g\circ g)a_0r_{X_0} + (\Delta a + a \circ f - g \circ ' a)(\overline{a_0}a_0 - 1).$$
Then we are left to notice that indeed
$$(\Delta g + g\circ g)(a_0r_{X_0}-r_{Y_0}a_0) = d((\Delta g+ g \circ g) r_{X_0Y_0})$$

and thus we have constructed the lift.
\end{proof}

Having Lemma \ref{technical} in our possession, we are left to notice that the degree $0$ components of the lift  are $a_i$, which are homotopy equivalences in $\A$ as $a$ was a homotopy equivalence in $M_nF(\A)$. Thus, we have proved the following theorem. 

\begin{theo}
\label{reedyfib}
For all $n$, matching maps $F_n(\A) \to M_nF(\A)$ are Dwyer-Kan fibrations. Consequently, $F_\bullet(\A)$ is Reedy fibrant.
\end{theo}

\begin{rem}
In \cite{Hols}, the Dwyer-Kan fibrancy of the matching maps was proved for the case when $\A$ is pretriangulated, by a strategy involving contraction of the cones. This strategy can be in fact performed in the case of arbitrary $\A$, which we demonstrate in Appendix \ref{app}.
\end{rem}

\begin{rem}
In the framework of $\infty$-local systems, the meaning of Reedy fibrancy is the following: if $a$ is a homotopy equivalence between two $\infty$-local systems on the simplex boundary, one of which was restricted from the simplex, then this homotopy equivalence can be lifted to a homotopy equivalence between two $\infty$-local systems on the simplex.
\end{rem}

\appendix 
\section{An alternative proof of Reedy fibrancy}
\label{app}

We now present a proof of Theorem \ref{reedyfib} that does not rely on Lemma \ref{technical}.

\subsection{Contraction of cones and pretriangulated envelopes}
We have to verify that lift of Proposition \ref{lift} is a homotopy equivalence in $F_n(\A)$. While an explicit computation might be possible, it appears to be very cumbersome even in the case $n=1$ (see \cite{Shoi}, Lemma 3.5). There exists, however, a strategy involving contractions of cones (see \cite{Tab2}).

\begin{defi}
For an object $X$ in some DG-category $\A$, its contraction is $b_X \in \A^{-1}(X,X)$ with $d(b_X)=1_X$. 
\end{defi}

Lemma \ref{klemma} precisely states that for any homotopy equivalence $\A$, you can construct a contraction of its cone in $\operatorname{Mod}\A$. However, one does not have to go as far as the whole category of DG-modules. Following \cite{Drinfeld}, recall the construction of the pretriangulated envelope.

\begin{defi}
For a DG-category $\A$, its pretriangulated envelope $\mathsf{Pretr}(\A)$ has one-sided twisted complexes as objects -- namely, those are formal expressions $(\bigoplus_{i=1}^n C_i[r_i],q)$, where $C_i$ are objects of $\A$, $r_i$ are integers and $q$ is a set of morphisms $q_{ij} \in (\A(C_j,C_i)[r_i-r_j])^1$ subject to $q_{ij}=0$ for $i \geq j$ and $dq + q \circ q = 0$. The morphisms are given by 
$$\mathsf{Pretr}(\A)((\bigoplus_{i=1}^n C_i[r_i],q),(\bigoplus_{j=1}^m C'_j[r'_j],q')) = \bigoplus_{i,j} \A(C_j,C'_i)[r'_i-r_j].$$
That is, a degree $k$ morphism $f\co (\bigoplus_{i=1}^n C_i[r_i],q) \to (\bigoplus_{i=1}^m C'_i[r'_i],q')$  is a set of components $f_{ij} \in (\A(C_j, C'_i)[r'_i-r_j])^k$, with matrix multiplication for composition and with differential given by
$$d_{TC}(f) = d(f)+q'\circ f-(-1)^k f \circ q.$$
\end{defi}

There are natural fully faithful embeddings $\A \hookrightarrow \mathsf{Pretr}(\A) \hookrightarrow \operatorname{Mod}\A$. For any $f \in Z^0(\A(X,Y))$, its cone is an object of $\mathsf{Pretr}(\A)$ defined as $\operatorname{Cone}(f)\co=(Y\oplus X[1],q)$ with $q_{12}=f$ (this is compatible with the embedding $\mathsf{Pretr}(\A) \hookrightarrow \operatorname{Mod}\A$). We say that $\A$ already has all the cones if $\operatorname{Cone}(f)$ is always isomorphic to some object in the image of the embedding $\A \hookrightarrow \mathsf{Pretr}(\A)$.  It can be checked that  $\mathsf{Pretr}(\A)$ has all the cones. \\

Note that for DG-categories that have all the cones, we can now prove the following lemma.

\begin{lem}
\label{pretr}
If $\A$ has all the cones, then the matching maps $m_n\co F_n(\A) \to M_nF(\A)$ are fibrations.
\end{lem}

\begin{proof}
We are left to check that if $\tilde{a}$ is a closed lift of a homotopy equivalence $a\co(X,f_{\leq n}) \to (Y,g)$, then $\tilde{a}$ is a homotopy equivalence in $F_n(\A)$. We notice that if $\A$ has all the cones, then $F_n(\A)$ and $M_nF(\A)$ also have all the cones. So $\operatorname{Cone}(a)$ is an object of $M_nF(\A)$ which (by Lemma \ref{klemma}) has a contraction $b$. Note that for any functor, the induced functor on pretriangulated envelopes respects cones, so $m_n(\operatorname{Cone}(\tilde{a})) = \operatorname{Cone}(a)$. Lifting $b$ to a contraction of $\operatorname{Cone}(\tilde{a})$ will then show that $\tilde{a}$ is a homotopy equivalence. And indeed, any contraction can be lifted along $m_n$. Let $b$ be a contraction of $(X,f_{\leq n})$. The the lift, as shown in \cite{Hols},  is obtained by setting
$$b_{0,\ldots, n} = b_0(\Delta b + b \circ f + f \circ b). $$
\end{proof}

We now show how the assumption of $\A$ having all the cones can be omitted. In \cite{Tab2}, this was done for the case $n=1$ via a quasiequivalence $\mathsf{Pretr}(F_1(\A)) \simeq F_1(\mathsf{Pretr}(\A))$.

\subsection{Proof of Theorem \ref{reedyfib}}

% Using Lemma \ref{emb}, consider the following commutative diagram:

% $$ \xymatrix{F_n(\A) \ar@{^(->}[r] \ar[d] & \mathsf{Pretr}(F_n(\A)) \ar@{^(->}[r] \ar[d]& F_n(\mathsf{Pretr}(\A))  \ar[d] \\
%M_nF(\A) \ar@{^(->}[r] & \mathsf{Pretr}(M_nF(\A)) \ar@{^(->}[r] & M_nF(\mathsf{Pretr}(\A))}
% $$

Consider the following commutative square, where the horizontal arrows are fully faithful embeddings given by compositions of $F_n$ (respectively $M_nF$) with natural embeddings $\A \hookrightarrow \mathsf{Pretr}(\A)$: 

$$\xymatrix{F_n(\A) \ar@{^(->}[r] \ar[d] & F_n(\mathsf{Pretr}(\A)) \ar[d]\\
M_nF (A) \ar@{^(->}[r] & M_nF(\mathsf{Pretr}(\A))}$$

For a homotopy equivalence $a\co (X,f_{\leq n}) \to (Y,g)$ in $M_nF(\A)$, we have constructed in Proposition \ref{lift} its closed lift along the left vertical arrow. Under embeddings, this is also a legitimate lift along the right vertical arrow. As the category $\mathsf{Pretr}(\A)$ has all the cones, we know from Lemma \ref{pretr} that any closed lift of a homotopy equivalence is a homotopy equivalence. So we are left to observe that embeddings respect homotopy equivalences, and that if a morphism is a homotopy equivalence in the larger category then it is also a homotopy equivalence in the smaller category. This concludes the proof.

\section{Erratum}
Contrary to what was stated in our paper at page 2, the proof of Proposition 3.3 in \cite{Tab2} did not contain any mathematical inaccuracies or flaws. Rather, it was a question of exposition: a trivial computation was omitted without explicitly mentioning the fact, and this led the author of \cite{Hols} to wrong conclusions. The proof in \cite{Hols} thus indeed contained a gap that we fixed.

\end{document}